\documentclass[a4paper,reqno,11pt]{amsart}
\usepackage{amsmath}
\usepackage{amssymb}
\usepackage{amsthm}
\usepackage{latexsym}
\usepackage[all]{xypic}
\usepackage{amsfonts}
\usepackage{mathrsfs}
\usepackage{graphicx,color}
\usepackage{xcolor}
\usepackage{tikz}
\usepackage{graphics}
\usetikzlibrary[shapes]

\usepackage{anysize,hyperref}
\usepackage{soul}
\setstcolor{red}

\newtheorem{theorem}{Theorem}[section]
\newtheorem{lemma}[theorem]{Lemma}
\newtheorem{corollary}[theorem]{Corollary}
\newtheorem{proposition}[theorem]{Proposition}
\theoremstyle{definition}
\newtheorem{definition}[theorem]{Definition}
\newtheorem{definition-proposition}[theorem]{Definition-Proposition}
\newtheorem{remark}[theorem]{Remark}
\newtheorem{notation}[theorem]{Notation}
\newtheorem{example}[theorem]{Example}

\def\C{\mathcal{C}}
\def\D{\mathcal{D}}
\def\H{\mathcal {H}}

\def\X{\mathscr{X}}
\def\Y{\mathscr{Y}}

\def\F{\mathcal {F}}

\def \text{\mbox}

\providecommand{\add}{\mathop{\rm add}\nolimits}%
\providecommand{\Ext}{\mathop{\rm Ext}\nolimits}%
\providecommand{\Hom}{\mathop{\rm Hom}\nolimits}%
\renewcommand{\mod}{\mathop{\rm mod}\nolimits}%
\providecommand{\ind}{\mathop{\rm ind}\nolimits}%

\def\XX{\widetilde{\X}}

\begin{document}

\title{Torsion pairs in repetitive cluster categories of type $A_n$}

\author[Chang]{Huimin Chang}
\address{
Department of Applied Mathematics
The Open University of China
100039 Beijing
China
}
\email{changhm@ouchn.edu.cn}
\begin{abstract}

We give a complete classification of torsion pairs in repetitive cluster categories of type $A_n$, which were defined by Zhu \cite{Z1} as the orbit categories $D^b(\mod KA_n)/\langle(\tau^{-1}[1])^p\rangle$, via certain configurations of diagonals, called Ptolemy diagrams. As applications, we classify rigid subcategories of $D^b(\mod KA_n)/\langle(\tau^{-1}[1])^p\rangle$, which gives Lamberti's classification of cluster tilting subcategories \cite{L1}. When $p=1$, this generalizes the work of Holm, J{\o}rgensen and Rubey for the classification of torsion pairs in cluster categories of type $A_{n}$ \cite{HJR1}.
\end{abstract}



\thanks{This work was supported by the NSF of China (Grant No.\;12301047). }

\maketitle

\tableofcontents
\section{Introduction}

Torsion theory is a classic subject in algebra. The notion of torsion pairs in abelian categories was first introduced by Dickson \cite{D}, and the triangulated version goes back to Iyama and Yoshino \cite{IY} to study cluster tilting subcategories in a triangulated category. Cotorsion pairs in a triangulated category were used by Nakaoka \cite{N} to unify the abelian structures arising from $t$-structures \cite{BBD} and from cluster tilting subcategories \cite{KR,KZ,IY}. Torsion pairs and cotorsion pairs in a triangulated category can be transformed into each other by shifting the torsion-free parts. Hence, classifications of torsion pairs and cotorsion pairs are equivalent in a triangulated category. 

In fact, there is a series of articles that have studied the classification of (co)torsion pairs in triangulated categories \cite{CZZ1,CZ,CZ1,GHJ,HJR1,HJR2,HJR3,Ng,ZZZ}. One of the important methods to classify (co)torsion pairs is to use the geometric model of the category. For example, Ng \cite{Ng} classified torsion pairs in the cluster category of type $A_\infty$ \cite{HJ2} via certain configurations of arcs of the infinity-gon. Holm, J{\o}rgensen and Rubey \cite{HJR1,HJR2,HJR3} classified torsion pairs in the cluster category of type $A_{n}$, in the cluster tube and in the cluster category of type $D_n$ via Ptolemy diagrams of a regular $(n+3)$-gon, periodic Ptolemy diagrams of the infinity-gon and Ptolemy diagrams of a regular $2n$-gon, respectively. Zhang, Zhou and Zhu \cite{ZZZ} classified cotorsion pairs in the cluster category of a marked surface via paintings of the surface. Chang and Zhu \cite{CZ} classified torsion pairs in certain finite $2$-Calabi-Yau triangulated categories with maximal rigid objects, via periodic Ptolemy diagrams of a regular polygon. Chang, Zhou and Zhu \cite{CZZ1} classified cotorsion pairs in the cluster categories of type $A^{\infty}_{\infty}$ via certain configurations of arcs of an infinite strip with marked points  in the plane, called $\tau$-compact Ptolemy diagrams. Gratz, Holm and J{\o}rgensen \cite{GHJ} gave a classification of torsion pairs in Igusa-Todorov's cluster categories of type $A_{\infty}$ via their combinatorial models. Recently, Chang and Zhu \cite{CZ1} classified cotorsion pairs in higher cluster categories of type $A_n$ via Ptolemy diagrams of a regular polygon.

Cluster categories defined by Buan, Marsh, Reineke, Reiten and Todorov \cite{BMRRT}, and also by Caldero, Chapoton and Schiffler for type $A_n$ \cite{CCS}, are the orbit categories $\D^b(H)/\langle F\rangle$ of derived categories $\D^b(H)$ of a hereditary algebra $H$ by the automorphism group generated by $F=\tau^{-1}[1]$, where $\tau$ is the Auslander-Reiten translation in $\D^b(H)$ and $[1]$ is the shift functor of $\D^b(H)$. In order to study cluster-tilted algebras and their intermediate coverings, Zhu \cite{Z1} introduced repetitive cluster categories, defined as the orbit categories $\D^b(\H)/\langle(\tau^{-1}[1])^p\rangle$ ($1\leq p\in\mathbb{N}$), where $\H$ is a hereditary abelian category with tilting objects. Furthermore, Lamberti \cite{L1} gave a geometric model of $\C_{n,p}\colon=D^b(\mod KA_n)/\langle(\tau^{-1}[1])^p\rangle$ by means of diagonals in the so-called repetitive polygon $\Pi^p$. This generalises the construction of Caldero, Chapoton and Schiffler for $p=1$ \cite{CCS}. The particularity of the repetitive cluster categories is that they are fractionally Calabi-Yau dimension $\frac{2p}{p}$, and the fraction cannot be simplified. It is precisely in this point that the category $\C_{n,p}$ differs from the case when $p=1$.

In this paper, we introduce the definition of Ptolemy diagrams of a regular $p(n+2)$-gon. We show that there is a bijection between the torsion pairs in $\C_{n,p}$ and the Ptolemy diagrams of a $p(n+2)$-gon. As applications, we classify rigid subcategories of $\C_{n,p}$,  which deduces Lamberti's classification of cluster tilting subcategories \cite{L1}. When $p=1$, this generalizes the work of Holm, J{\o}rgensen and Rubey for the classification of torsion pairs in cluster categories of type $A_{n}$ \cite{HJR1} .

The paper is organized as follows. In Section 2, we review the notions of $\C_{n,p}$, and its geometric realization in Section 3. In Section 4, we define Ptolemy diagrams of a $p(n+2)$-gon and give a classification of torsion pairs in $\C_{n,p}$. Some applications are given in the last section.

\subsection{Notations} 
Throughout this paper, $K$ stands for an algebraically closed field. Any subcategory of a category is assumed to be full and closed under taking isomorphisms, finite direct sums and direct summands. In this way, any subcategory is determined by the indecomposable objects in this subcategory. For two subcategories $\X,\Y$ of a triangulated category $\C$, $\Hom_{\C}(\X,\Y)=0$ means $\Hom_{\C}(X,Y)=0$ for any $X\in \X$ and any $Y\in \Y$. For two subcategories $\X,\Y$ of a triangulated category $\C$, denote by $\X\ast\Y$  the  subcategory of $\C$ whose objects are $M$ which fits into a triangle $$X\to M\to Y\to X[1]$$ with $X\in\X$ and $Y\in\Y$. A subcategory $\X$ is called  extension closed  provided that $\X\ast\X \subseteq\X.$ For an object $M\in\C$, let $\add M$ denote the additive closure of $M$. We denote by $\X^\perp$ (resp. $^\perp\X$) the subcategory whose objects are $M\in\C$ satisfying $\Hom_{\C}(\X,M)=0$ (resp. $\Hom_{\C}(M,\X)=0$). We use $\Ext_{\C}^1(X,Y)$ to denote $\Hom_{\C}(X,Y[1])$, where $[1]$ is the shift functor of $\C$. Throughout the paper, we denote by $\D$ the category of $\D^b(\mod KA_n)$.

\section{Repetitive cluster categories of type $A_{n}$}
\subsection{Serre duality and Calabi-Yau categories}
\begin{definition}
A $K$-linear triangulated category $\C$ has a Serre functor if it is equipped with an auto-equivalence $\nu\colon \C\rightarrow \C$ together with bifunctorial isomorphisms
$$D\Hom_{\C}(X,Y)\cong\Hom_{\C}(Y,\nu X),$$
for each $X,Y\in\C$. Moreover, if $\C$ admits a Serre functor, we say that $\C$ has Serre duality. 
\end{definition}
Note that for a hereditary abelian category $\H$, if $\H$ has tilting objects, then $\H$ has Serre duality, and also Auslander-Reiten translate $\tau$ (shortly AR-translate) \cite{HRS}. Thus, a Serre functor $\nu$ of $\D$ exists, it is unique up to isomorphism and $\nu=\tau[1]$.
\begin{definition}
A triangulated category $\C$ with Serre functor $\nu$ is called $\frac{m}{n}$-Calabi-Yau ($\frac{m}{n}$-CY) for $n,m>0$ provided there is an isomorphism of triangle functors 
$$\nu^n\cong [m]$$
\end{definition}
Note that a $\frac{m}{n}$-CY category is also of $\frac{mt}{nt}$-CY, $t\in\mathbb{Z}$. However, the converse is not always true, this means the fraction cannot be simplified in general.
\subsection{Orbit categories}
Let $\C$ be a triangulated category and $F\colon \C\rightarrow \C$ be an autoequivalence. The orbit category $\mathcal O_{F}:= \C/\langle F\rangle$ has the same objects as $\C$ and its morphisms from $X$ to $Y$ are in bijection with 
$$\bigoplus_{i\in\ \mathbb{Z}} \mathrm{Hom}_{\C}(X,F^{i}Y).$$ 
In general, the orbit category is not triangulated \cite{K}. However, there is a sufficient set of conditions such that the orbit category is triangulated. Let $\C=D^b(\H)$ be the bounded derived category of a hereditary abelian $K$-category $\H$, and $F\colon \C\rightarrow \C$ be an autoequivalence.  Assume that the following hypotheses hold:
\begin{itemize}
  \item [(g1)] For each indecomposable $U$ of $\H$, only finitely many objects $F^iU, i\in\mathbb{Z}$,
lie in $\H$.
  \item [(g2)] There is an integer $N\geq 0$ such that the $F$-orbit of each indecomposable of
$\C$ contains an object $U[n]$, for some $0\leq n\leq N$ and some indecomposable object $U$ of $\H$.
\end{itemize}
Then the orbit category $\D/\langle F\rangle$ admits a natural triangulated structure such that
the projection functor $\pi\colon\D\rightarrow\D/\langle F\rangle$ is triangulated. In particular, the orbit category $\D/\langle F\rangle$ is called a cluster category when $F=\tau^{-1}[1]$, and the orbit category $\D/\langle F\rangle$ is called a repetitive cluster category when $F=(\tau^{-1}[1])^p$. 
\subsection{Repetitive cluster categories $\C_{n,p}$}
Throughout the paper, we study orbit categories of the form
$$\C_{n,p}=\D/\langle(\tau^{-1}[1])^p\rangle.$$
We call $\C_{n,p}$ the repetitive cluster categories of type $A_n$. The class of objects is the same as the class of objects in $\D$ and the space of morphisms is given by 
$$\Hom_{\C_{n,p}}(X,Y)=\bigoplus_{i\in\ \mathbb{Z}} \Hom_{\D}(X,(\tau^{-p}[p])^{i}Y).$$
Observe that when $p=1$, one gets the classical cluster category which we simply denote by $\C_{n}$. Furthermore, there is a natural projection functor $\pi_p\colon\D\rightarrow\C_{n,p}$, and the projection functor $\pi_p$ is simply denoted by $\pi$ if $p=1$. Moreover, one can define the projection functor $\eta_p\colon\C_{n,p}\rightarrow\C_{n}$ which sends an object $X$ in $\C_{n,p}$ to an object $X$ in $\C_{n}$, and $\phi\colon X\rightarrow Y$ in $\C_{n,p}$ to the morphism $\phi\colon X\rightarrow Y$ in $\C_{n}$. It is easy to check that $\pi=\eta_p\circ\pi_p$, i.e. the following diagram commutes:
$$
{
\xymatrix@-7mm@C-0.01cm{
     &&\D \ar[rdd]_{\pi_p}\ar[rr]^{\pi} & & \C_{n} \\
 \\
 &&&\C_{n,p} \ar[ruu]_{\eta_p} \\
 \\
}
}
$$
Thus, the repetitive cluster categories $\C_{n,p}$ serve as intermediate categories between the cluster categories $\C_{n}$ and derived categories $\D$. In the following, we summarize some known facts about $\C_{n,p}$ from \cite{L1}. 
\begin{figure}
\begin{center}
\begin{tikzpicture}
%\node (F) [draw,trapezium,minimum width=3cm]{F};
\draw (-3,-1)-- (0,-1)--(-1,1)--(-2,1)--(-3,-1);
\draw (-0.8,1)--(2.2,1)--(1.2,-1)--(0.2,-1)--(-0.8,1);
\node at (2.4,0){$\cdots$};
\node at (-1.5,0){$\F$};
\node at (0.7,0){$\F_2$};
\node at (4.2,0){$\F_{p-1}$};
\node at (6.4,0){$\F_{p}$};
\draw (2.7,1)--(5.7,1)--(4.7,-1)--(3.7,-1)--(2.7,1);
\draw (4.9,-1)-- (7.9,-1)--(6.9,1)--(5.9,1)--(4.9,-1);
\end{tikzpicture}
\end{center}
\caption{Partition of the fundamental domain of $\tau^{-p}[p]$ for an odd value $p$}
\label{01}
\end{figure}
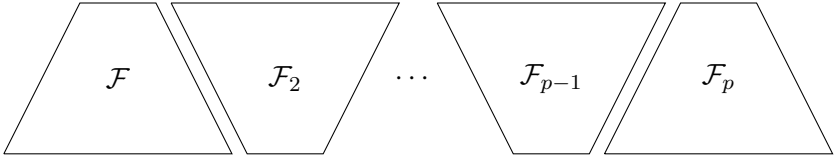

Let $\F=\F_1$ be the fundamental domain for the functor $\tau^{-1}[1]$ in $\D$ given by the isoclasses of indecomposable objects in $\C_n$, and $\F_i$ be the ($\tau^{-1}[1])^{i-1}$-shifts of $\F$ with $1\leq i\leq p$. We can draw the fundamental domain for the functor $\tau^{-p}[p]$ as in Fig. \ref{01}.
\begin{proposition}\cite[Lemma 2.7]{L1} The following statements hold for $\C_{n,p}$.
\begin{enumerate}
\item The projection functors $\pi_p\colon\D\rightarrow\C_{n,p}$ and $\eta_p\colon\C_{n,p}\rightarrow\C_{n}$ are trialgle functors.
\item The category $\C_{n,p}$ is triangulated with Serre functor $\tau[1]$ induced from $\D$.
\item The category $\C_{n,p}$ is a Krull-Schmidt and $\frac{2p}{p}$-CY category, i.e.
$$(\tau[1])^p\cong [2p].$$
\item $\ind(\C_{n,p})=\bigcup_{i=1}^{p}(\F_i)$.
\end{enumerate}
\end{proposition}
\section{Geometric model of $\C_{n,p}$}
In this section, we recall from \cite{L1} a geometric description of $\C_{n,p}$. The geometric model for $\C_{n,p}$ can be viewed as a $p$-covering of the polygon modeling $\C_{n}$ \cite{CCS}. Let $\Pi^p$ be a regular $p(n+2)$-gon, where $n\in\mathbb{N}, p>1$, and $\Pi$ be a $N\colon=n+3$-gon throughout this paper.
\subsection{The repetitive polygon $\Pi^p$}
We number the vertices of $\Pi^p$ clockwise repeating $p$ times the $N$-tuple $1,2,\ldots, N-1,N$ and let correspongding $N\equiv 1$ if $p>1$. For the case $p=1$, let $\Pi^1=\Pi$. Then we denote by $\Pi_1$ a region homotopic to $\Pi$ inside $\Pi^p$ determined by the segments $(1,2),(2,3),\ldots,(N-1,N)$ and the inner arc $(1,N)$. 

Denote by $\rho\colon\Pi^p\rightarrow\Pi^p$ the clockwise rotation through $\frac{2\pi}{p}$ around the center of $\Pi^p$, and set $\Pi_i=\rho^{i-1}(\Pi_1)$ for $1\leq i\leq p$. In this way $\Pi^p$ is divided into $p$ regions. See Fig. \ref{02} for an illustration.
\begin{figure}
\begin{center}
\begin{tikzpicture}[scale=3]
\begin{scope}
    \foreach \x in {-24,0,24,48,72,96,120,144,168,192,216,240,264,288,312}{
        \draw (\x:1 cm) -- (\x + 24: 1cm) -- cycle;
        %\draw (72:1cm)--(252:1cm);
        \draw (72:1cm)..controls(-17:0.2cm)..(312:1cm);
        \draw (72:1cm)..controls(110:0.2cm)..(192:1cm);
        \draw (192:1cm)..controls(240:0.2cm)..(312:1cm);
         \node at (24:1.1cm) {3};
         \node at (48:1.1cm) {2};
         \node at (72:1.1cm) {1=6};
         \node at (0:1.1cm) {4};
         \node at (-24:1.1cm) {5};
         \node at (312:1.1cm) {1=6};
         \node at (288:1.1cm) {2};
         \node at (264:1.1cm) {3};
         \node at (240:1.1cm) {4};
          \node at (216:1.1cm) {5};
         \node at (192:1.15cm) {1=6};
         \node at (168:1.1cm) {2};
          \node at (144:1.1cm) {3};
         \node at (120:1.1cm) {4};
         \node at (96:1.15cm) {5};
         \node at (0:0.7cm) {\huge${\Pi_1}$};
         \node at (120:0.7cm) {\huge${\Pi_3}$};
         \node at (250:0.7cm) {\huge${\Pi_2}$};
}
\end{scope}
\end{tikzpicture}
\caption{The polygon $\Pi^3$ for $n=3$}
\label{02}
\end{center}
\end{figure}
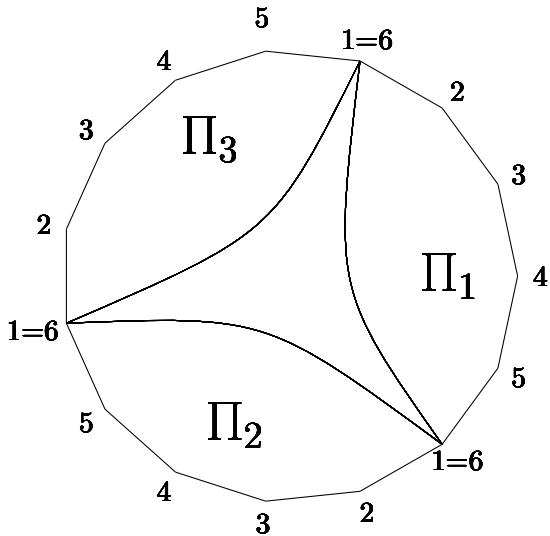

Recall that a diagonal of $\Pi$ is a line between two non-adjacent vertices $i$ and $j$ of $\Pi$, denoted by $(i,j)$.
\begin{definition}
The diagonals of $\Pi^p$ are given by the union of all diagonals of 
$$\Pi_1\cup\Pi_2\cup\ldots\cup\Pi_p.$$
\end{definition}
Denote the diagonals of $\Pi^p$ by the tripe $(i,j,k)$, where $k$ specifies a region $\Pi_k$ inside $\Pi^p$ with $1\leq k\leq p$, and the tuple $(i,j)$ is the diagonal in $\Pi_k$. Note that $(1,N,k)$ for $1\leq k\leq p$ are not diagonals of $\Pi^p$ as they correspond to boundary segments of $\Pi_k$.
\begin{notation}
In the writing of a diagonal $(i,j,k)$ of $\Pi^p$, we understand that the indices $i,j$ have to be taken modulo $N$, and the index $k$ modulo $p$. We always assume $i<j$ in this paper. 
\end{notation}
\subsection{Translation quiver of diagonals of $\Pi^p$}
In order to give a geometric model of the AR-quiver of $\C_{n,p}$, we associate a translation quiver to the diagonals of $\Pi^p$. 
\begin{definition}
Let $\Gamma_{n,p}$ be the quiver of diagonals of $\Pi^p$, whose vertices are the diagonals of $\Pi^p$, and whose arrows are defined as follows.
\begin{enumerate}
  \item For $j\neq N$, 
  $${
\xymatrix@-7mm@C-0.01cm{
     &&(i,j+1,k), \\
 (i,j,k)\ar[rru]\ar[rrd]\\
 &&(i+1,j,k).\\
 \\
}
}
  $$
  \item For $j=N$, $(i,N,k)\rightarrow (i+1,N,k)$.
  \item Furthermore, $(i,N,k)\rightarrow (1,i,k+1)$.
\end{enumerate}
\end{definition}
\begin{remark}
In order to make readers better understanding the definition of arrows, we give an interpretation. Observe that the description (1) and (2) is compatible with the arrow defined in \cite{CCS}. That is, we draw an arrow from $\alpha$ to $\beta$ if $\alpha$ and $\beta$ have a common vertex and the other vertex of $\alpha$ can be rotated clockwise to $\beta$ by one vertex inside $\Pi_k$. The description (3) means that we draw an arrow to the diagonal we can reach with a clockwise rotation around vertex $i$ by one vertex inside $\Pi_k$ composed with $\rho$, i.e. $(i,N,k)$ is linked to $\rho(1,i,k)=(1,i,k+1)$.
\end{remark}
Now we equip $\Gamma_{n,p}$ with a translation $\tau\colon\Gamma_{n,p}\rightarrow\Gamma_{n,p}$ such that $(\Gamma_{n,p},\tau)$ becomes a stable translation quiver.
\begin{definition}
The translation $\tau$ acts on diagonals of $\Pi^p$ is as follows:
\begin{itemize}
  \item If $i\neq 1$, $\tau(i,j,k)=(i-1,j-1,k)$.
  \item If $i=1$, $\tau(1,j,k)=\rho^{-1}(j-1,N,k)=(j-1,N,k-1)$.
\end{itemize}
\end{definition}
Then the pair $(\Gamma_{n,p}, \tau)$ is a stable translation quiver by \cite{L1}.
\subsection{Equivalence of $\Gamma_{n,p}$ and $\C_{n,p}$}
In this subsection, we recall the equivalence of $\Gamma_{n,p}$ and $\C_{n,p}$ from \cite{L1}. 
\begin{lemma}\cite[Proposition 3.7]{L1}\label{lem1}
There is a bijection between indecomposable objects in $\C_{n,p}$ and the diagonals of $\Pi^p$.
\end{lemma}
In the following, we always use a diagonal $(i,j,k)$ to represent an indecomposable object $M$ in $\C_{n,p}$ without confusion, denote by $M=(i,j,k)$. This bijection induces a bijection between the subcategories of $\C_{n,p}$ and the sets of diagonals in $\Pi^p$. For a subcategory $\X$ of $\C_{n,p}$, we denote the corresponding set of diagonals in $\Pi^p$ by $\XX$. 

The correspondence induces the action of the shift functor $[1]$ in $\C_{n,p}$ on diagonals of $\Pi^p$, which is as follows:
\begin{itemize}
  \item If $i=1$, then $(i,j,k)[1]=(j-1,N,k)$.
  \item If $i\neq1$, then $(i,j,k)[1]=\rho(i-1,j-1,k)=(i-1,j-1,k+1)$.
\end{itemize}

Remember that the fundamental domain $\F$ is given by the isoclasses of indecomposable objects in $\C_n$, which corresponds to diagonals in $\Pi$. For two diagonals $X=(i,j,\ell)$ and $Y=(i^\prime,j^\prime,\ell)$ in $\Pi_{\ell}$, $X$ and $Y$ cross if $i<i^\prime<j<j^\prime$ or $i^\prime<i<j^\prime<j$. Recall the fact from \cite[Sec. 5]{CCS} that $\Ext_{\C_{n}}^{1}(X,Y)\neq 0$ if and only if the two diagonals $X$ and $Y$ cross. The following result is useful for classification of torsion pairs in $\C_{n,p}$. 
\begin{lemma}\cite[Lemma 5.3]{L1}\label{lem2}
Let $X=(i,j,\ell)$ and $Y=(i^\prime,j^\prime,\ell^\prime)$ be two diagonals in $\Pi^p$. We have
\[\Ext_{\C_{n,p}}^{1}(X,Y)\neq 0\Leftrightarrow\left\{
\begin{array}{cc}
\ell=\ell^\prime\text{\ and\ } 1\leq i^\prime<i<j^\prime<j\leq N,\\
\ell=\ell^\prime+1\text{\ and\ }1\leq i<i^\prime<j<j^\prime\leq N.
\end{array}
\right.
\]
Dually, we have 
\[\Ext_{\C_{n,p}}^{1}(Y,X)\neq 0\Leftrightarrow\left\{
\begin{array}{cc}
\ell=\ell^\prime\text{\ and\ } 1\leq i<i^\prime<j<j^\prime\leq N,\\
\ell^\prime=\ell+1\text{\ and\ }1\leq i^\prime<i<j^\prime<j\leq N.
\end{array}
\right.
\]
Otherwise $\Ext_{\C_{n,p}}^{1}(-,-)=0$.
\end{lemma}
\begin{example}
Let $n=3,p=3$. We draw the AR-quiver of $\C_{3,3}$ by Lemma \ref{lem1}, see Fig. \ref{03}. Let $X=(2,5,1)$. By Lemma \ref{lem2}, we get the following results. 
\begin{itemize}
  \item [(1)] The indecomposable objects $Y$ that satisfy $\Ext_{\C_{3,3}}^{1}(X,Y)\neq 0$ are exactly $(3,6,3),(4,6,3),(1,3,1),(1,4,1)$. 
  \item [(2)] The indecomposable objects $Z$ that satisfy $\Ext_{\C_{3,3}}^{1}(Z,X)\neq 0$ are exactly $(3,6,1),(4,6,1),(1,3,2),(1,4,2)$. 
\end{itemize}
\begin{center}
\begin{figure}
\begin{tikzpicture}[scale=0.8,
fl/.style={->,shorten <=6pt, shorten >=6pt,>=latex}]
%%%%%%%%%%%%%%%%%%%%%
% Coordinates
%%%%%%%%%%%%%%%%%%%%%
\coordinate (13) at (0,0) ;
\coordinate (14) at (1,1) ;
\coordinate (15) at (2,2) ;
\coordinate (24) at (2,0) ;
\coordinate (25) at (3,1) ;
\coordinate (26) at (4,2) ;
\coordinate (35) at (4,0) ;
\coordinate (36) at (5,1) ;
\coordinate (37) at (6,2) ;
\coordinate (46) at (6,0) ;
\coordinate (47) at (7,1) ;
\coordinate (48) at (8,2) ;
\coordinate (57) at (8,0) ;
\coordinate (58) at (9,1) ;
\coordinate (59) at (10,2) ;
\coordinate (68) at (10,0) ;
\coordinate (69) at (11,1) ;
\coordinate (610) at (12,2) ;
\coordinate (79) at (12,0) ;
\coordinate (710) at (13,1) ;
\coordinate (711) at (14,2) ;
\coordinate (810) at (14,0) ;
\coordinate (811) at (15,1) ;
\coordinate (812) at (16,2) ;
\coordinate (911) at (16,0) ;
\coordinate (912) at (17,1) ;
\coordinate (913) at (18,2) ;
\coordinate (1012) at (18,0) ;
\coordinate (1013) at (19,1) ;
\coordinate (1014) at (20,2) ;
\coordinate (1113) at (20,0) ;
\coordinate (1114) at (21,1) ;
\coordinate (1115) at (22,2) ;
%%%%%%%%%%%%%%%%%%%%%
% Arrows
%%%%%%%%%%%%%%%%%%%%%

\draw[fl] (13) -- (14) ;
\draw[fl] (14) -- (15) ;
\draw[fl] (14) -- (24) ;
\draw[fl] (15) --(25) ;
\draw[fl] (24) --(25) ;
\draw[fl] (25) --(35) ;
\draw[fl] (25) --(26) ;
\draw[fl] (35) --(36) ;
\draw[fl] (36) --(37) ;
\draw[fl] (26) --(36) ;
\draw[fl] (46) --(47) ;
\draw[fl] (47) --(48) ;
\draw[fl] (37) --(47) ;
\draw[fl] (36) --(46) ;
\draw[fl] (57) --(58) ;
\draw[fl] (58) --(59) ;
\draw[fl] (48) --(58) ;
\draw[fl] (47) --(57) ;
\draw[fl] (58) --(68) ;
\draw[fl] (68) --(69) ;
\draw[fl] (69) --(610) ;
\draw[fl] (59) --(69) ;
\draw[fl] (79) --(710) ;
\draw[fl] (710) --(711) ;
\draw[fl] (610) --(710) ;
\draw[fl] (69) --(79) ;
\draw[fl] (810) --(811) ;
\draw[fl] (811) --(812) ;
\draw[fl] (711) --(811) ;
\draw[fl] (710) --(810) ;
\draw[fl] (911) --(912) ;
\draw[fl] (912) --(913) ;
\draw[fl] (812) --(912) ;
\draw[fl] (811) --(911) ;
\draw[fl] (1012) --(1013) ;
%\draw[fl] (1013) --(1014) ;
\draw[fl] (913) --(1013) ;
\draw[fl] (912) --(1012) ;
%\draw[fl] (1113) --(1114) ;
%\draw[fl] (1114) --(1115) ;
%\draw[fl] (1014) --(1114) ;
\draw[fl] (1013) --(1113) ;
\draw (13) node[scale=0.5] {(1,3,1)} ;
\draw (14) node[scale=0.5] {(1,4,1)} ;
\draw (15) node[scale=0.5] {(1,5,1)} ;
\draw (24) node[scale=0.5] {(2,4,1)} ;
\draw (25) node[scale=0.5] {(2,5,1)} ;
\draw (26) node[scale=0.5] {(2,6,1)} ;
\draw (35) node[scale=0.5] {(3,5,1)} ;
\draw (36) node[scale=0.5] {(3,6,1)} ;
\draw (37) node[scale=0.5] {(1,3,2)} ;
\draw (46) node[scale=0.5] {(4,6,1)} ;
\draw (47) node[scale=0.5] {(1,4,2)} ;
\draw (48) node[scale=0.5] {(2,4,2)} ;
\draw (57) node[scale=0.5] {(1,5,2)} ;
\draw (58) node[scale=0.5] {(2,5,2)} ;
\draw (59) node[scale=0.5] {(3,5,2)} ;
\draw (68) node[scale=0.5] {(2,6,2)} ;
\draw (69) node[scale=0.5] {(3,6,2)} ;
\draw (610) node[scale=0.5] {(4,6,2)} ;
\draw (79) node[scale=0.5] {(1,3,3)} ;
\draw (710) node[scale=0.5] {(1,4,3)} ;
\draw (711) node[scale=0.5] {(1,5,3)} ;
\draw (810) node[scale=0.5] {(2,4,3)} ;
\draw (811) node[scale=0.5] {(2,5,3)} ;
\draw (812) node[scale=0.5] {(2,6,3)} ;
\draw (911) node[scale=0.5] {(3,5,3)} ;
\draw (912) node[scale=0.5] {(3,6,3)} ;
\draw (913) node[scale=0.5] {(1,3,1)} ;
\draw (1012) node[scale=0.5] {(4,6,3)} ;
\draw (1013) node[scale=0.5] {(1,4,1)} ;
\draw (1113) node[scale=0.5] {(1,5,1)} ;
\draw[thick, dashed, blue] (-1,-0.5) -- (7,-0.5) -- (4,2.5) -- (1.5,2.5) --cycle ;
\draw[thick, dashed, blue] (8,-0.5) -- (10.5,-0.5) -- (13,2.5) -- (5,2.5) --cycle ;
\draw[thick, dashed, blue] (11,-0.5) -- (19.5,-0.5) -- (16.5,2.5) -- (14,2.5) --cycle ;
\end{tikzpicture}
\caption{The AR-quiver of $\C_{3,3}$. The fundamental domains $\F_1,\F_2$ and $\F_3$ are encircled by dashed lines}
\label{03}
\end{figure}
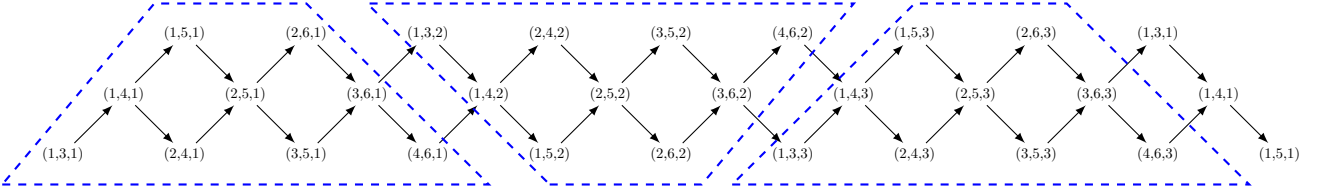
\end{center}
\end{example}
\section{Classification of torsion pairs in $\C_{n,p}$}
In this section, we define Ptolemy diagrams in $\Pi^p$ and give a complete classification of torsion pairs in $\C_{n,p}$. Firstly, we summarize some known facts about torsion pairs in a triangulated category based on \cite{BMRRT,IY,N}.
\subsection{Definition of torsion pairs}
\begin{definition}
Let $\X, \Y$ be subcategories of a triangulated category $\C$.
\begin{itemize}
  \item [(1)] The pair $(\X, \Y)$ is called a torsion pair if
  $$\Hom_{\C}(\X, \Y)=0\;\text{and}\;\C=\X\ast\Y. $$
  \item [(2)] The pair $(\X, \Y)$ is called a cotorsion pair if
  $$\Ext^{1}_{\C}(\X, \Y)=0\;\text{and}\;\C=\X\ast\Y[1]. $$
\end{itemize}
\end{definition}
\begin{remark}\label{remark3}
$(\X,\Y)$ is a cotorsion pair if and only if $(\X,\Y[1])$ is a torsion pair.
\end{remark}
\begin{lemma}\cite[Proposition 2.3]{IY},\cite[Proposition 2.3]{HJR1}\label{lemma0}
Let $\X$ be a subcategory of $\C_{n,p}$. Then the following statements are equivalent.
\begin{itemize}
  \item [(1)] $\X$ is closed under extensions.
  \item [(2)] $(\X,\X^\perp)$ is a torsion pair.
  \item [(3)] $\X={^\perp(\X^\perp)}$.
\end{itemize}
\end{lemma}
\subsection{Ptolemy diagrams in $\Pi^p$}
Before we give the definition of Ptolemy diagrams, we define when two diagonals of $\Pi^p$ cross.
\begin{definition}\label{cross}
Two diagonals $X=(i,j,\ell)$ and $Y=(i^\prime,j^\prime,\ell^\prime)$ of $\Pi^p$ are said to be cross if one of the following conditions is satisfied:
\begin{itemize}
  \item [(1)] $\ell=\ell^\prime$ and $1\leq i<i^\prime<j<j^\prime\leq N$ or $1\leq i^\prime<i<j^\prime<j\leq N$.
  \item [(2)] $\ell^\prime=\ell+1$ and  $1\leq i^\prime<i<j^\prime<j\leq N$.
  \item [(3)] $\ell=\ell^\prime+1$ and  $1\leq i<i^\prime<j<j^\prime\leq N$.
\end{itemize}
\begin{remark}
By Lemma \ref{lem2}, $X=(i,j,\ell)$ crosses $Y=(i^\prime,j^\prime,\ell^\prime)$ if and only if $\Ext^1_{\C_{n,p}}(X,Y)\neq 0$ or $\Ext^1_{\C_{n,p}}(Y,X)\neq 0$. When $p=1$, the conditions in Definition \ref{cross} reduce to (1) only, which is compatible with the definition of cross in $\Pi$.
\end{remark}
\end{definition}
\begin{definition}\label{c2}
  Let $\mathfrak{U}$  be a set of diagonals in $\Pi^p$. Then $\mathfrak{U}$ is called a
  Ptolemy diagram if it satisfies the following
  conditions.  Let $X=(i,j,\ell)$ and $Y=(i^\prime,j^\prime,\ell^\prime)$
  be crossing diagonals in $\mathfrak{U}$ (in the sense of
  Definition \ref{cross}).
  \begin{enumerate}
  \item[{(Pt1)}] If $\ell=\ell^\prime$ with $1\leq i<i^\prime<j<j^\prime\leq N$ or $1\leq i^\prime<i<j^\prime<j\leq N$, i.e. $X$ and $Y$  cross in $\Pi_\ell$, we have the following two cases:
      \begin{itemize}
        \item [(i)] If $1\leq i<i^\prime<j<j^\prime\leq N$, then those of $(i^\prime,j,\ell)$, $(i,j^\prime,\ell)$ which are diagonals of $\Pi^p$ are also in $\mathfrak{U}$.
        \item [(ii)] If $1\leq i^\prime<i<j^\prime<j\leq N$, then those of $(i,j^\prime,\ell)$, $(i^\prime,j,\ell)$ which are diagonals of $\Pi^p$ are also in $\mathfrak{U}$.
      \end{itemize}
  \item[{(Pt2)}] If $\ell^\prime=\ell+1$ with $1\leq i^\prime<i<j^\prime<j\leq N$,
  then those of $(i^\prime,i,\ell^\prime)$, $(j^\prime,j,\ell)$ which are diagonals of $\Pi^p$ are also in $\mathfrak{U}$.
  \item[{(Pt3)}] If $\ell=\ell^\prime+1$ with $1\leq i<i^\prime<j<j^\prime\leq N$,
  then those of $(i,i^\prime,\ell)$, $(j,j^\prime,\ell^\prime)$ which are diagonals of $\Pi^p$ are also in $\mathfrak{U}$.
  \end{enumerate}
\end{definition}
These conditions are illustrated in Fig. \ref{04}.
\begin{figure}
%\label{table:Ptolemy}
  \centering
  \begin{enumerate}
  \item[{(Pt1)}]
  %\label{item:cond1}
  The first Ptolemy condition:
$$
\begin{tikzpicture}[scale=3]
\begin{scope}
    \foreach \x in {-20,0,20,40,60,80,100,120,140,160,180,200,220,240,260,280,300,320}{
        \draw (\x:1 cm) -- (\x + 20: 1cm) -- cycle;
        \draw [thick](20:1cm)--(-40:1cm);
        \draw [thick](0:1cm)--(60:1cm);
        %\draw [thick][dashed](20:1cm)--(60:1cm);
        \draw [thick][dashed](60:1cm)--(-40:1cm);
        %\draw [thick][dashed](0:1cm)--(-40:1cm);
        \draw [thick][dashed](20:1cm)..controls(10:0.97cm)..(0:1cm);
        \node at (-40:1.1cm) {$j^\prime$};
        \node at (20:1.1cm) {$i^\prime$};
        \node at (0:1.1cm) {$j$};
        \node at (60:1.1cm) {$i$};
 \draw (320:1cm)..controls(180:-0.1cm)..(80:1cm); 
  \draw (80:1cm)..controls(120:0.2cm)..(200:1cm);  
   \draw (200:1cm)..controls(240:0.2cm)..(320:1cm); 
   \node at (180:-0.4cm){\huge{$\Pi_\ell$}};       
        }
\end{scope}
\end{tikzpicture}
\begin{tikzpicture}[scale=3][auto]
\begin{scope}
    \foreach \x in {-20,0,20,40,60,80,100,120,140,160,180,200,220,240,260,280,300,320}{
        \draw (\x:1 cm) -- (\x + 20: 1cm) -- cycle;
        \draw [thick](20:1cm)--(-40:1cm);
        \draw [thick](0:1cm)--(60:1cm);
        %\draw [thick][dashed](20:1cm)--(60:1cm);
        \draw [thick][dashed](60:1cm)--(-40:1cm);
        %\draw [thick][dashed](0:1cm)--(-40:1cm);
        \draw [thick][dashed](20:1cm)..controls(10:0.97cm)..(0:1cm);
        \node at (-40:1.1cm) {$j$};
        \node at (20:1.1cm) {$i$};
        \node at (0:1.1cm) {$j^\prime$};
        \node at (60:1.1cm) {$i^\prime$};
 \draw (320:1cm)..controls(180:-0.1cm)..(80:1cm); 
  \draw (80:1cm)..controls(120:0.2cm)..(200:1cm);  
   \draw (200:1cm)..controls(240:0.2cm)..(320:1cm); 
   \node at (180:-0.4cm){\huge{$\Pi_\ell$}};       
        }
\end{scope}
\end{tikzpicture}
$$

  \item[{(Pt2)}]
The second Ptolemy condition:
  $$
\begin{tikzpicture}[scale=3]
\begin{scope}
    \foreach \x in {-20,0,20,40,60,80,100,120,140,160,180,200,220,240,260,280,300,320}{
        \draw (\x:1 cm) -- (\x + 20: 1cm) -- cycle;
        \draw [thick](20:1cm)--(-40:1cm);
        \draw [thick](0:1cm)--(60:1cm);
        \draw [thick](260:1cm)--(200:1cm);
        \draw [thick](240:1cm)--(300:1cm);
        \draw [thick][dashed](260:1cm)--(300:1cm);
        \draw [thick][dashed](0:1cm)--(-40:1cm);
        \node at (-40:1.1cm) {$j$};
        \node at (20:1.1cm) {$i$};
        \node at (0:1.1cm) {$j^\prime$};
        \node at (60:1.1cm) {$i^\prime$};
          \node at (200:1.1cm) {$j$};
        \node at (260:1.1cm) {$i$};
        \node at (240:1.1cm) {$j^\prime$};
        \node at (300:1.1cm) {$i^\prime$};
 \draw (320:1cm)..controls(180:-0.1cm)..(80:1cm); 
  \draw (80:1cm)..controls(120:0.2cm)..(200:1cm);  
   \draw (200:1cm)..controls(240:0.2cm)..(320:1cm); 
   \node at (180:-0.4cm){\huge{$\Pi_\ell$}};    
   \node at (100:-0.6cm){\huge{$\Pi_{\ell+1}$}};    
        }
\end{scope}
\end{tikzpicture}
$$
%\caption{The case $m$ is even and odd, respectively}
\item[{(Pt3)}]
%\label{item:cond3}
The third Ptolemy condition:
\[
\begin{tikzpicture}[scale=3]
\begin{scope}
    \foreach \x in {-20,0,20,40,60,80,100,120,140,160,180,200,220,240,260,280,300,320}{
        \draw (\x:1 cm) -- (\x + 20: 1cm) -- cycle;
        \draw [thick](20:1cm)--(-40:1cm);
        \draw [thick](0:1cm)--(60:1cm);
        \draw [thick](260:1cm)--(200:1cm);
        \draw [thick](240:1cm)--(300:1cm);
        \draw [thick][dashed](-40:1cm)--(0:1cm);
        \draw [thick][dashed](260:1cm)--(300:1cm);
        \node at (-40:1.1cm) {$j^\prime$};
        \node at (20:1.1cm) {$i^\prime$};
        \node at (0:1.1cm) {$j$};
        \node at (60:1.1cm) {$i$};
          \node at (200:1.1cm) {$j^\prime$};
        \node at (260:1.1cm) {$i^\prime$};
        \node at (240:1.1cm) {$j$};
        \node at (300:1.1cm) {$i$};
 \draw (320:1cm)..controls(180:-0.1cm)..(80:1cm); 
  \draw (80:1cm)..controls(120:0.2cm)..(200:1cm);  
   \draw (200:1cm)..controls(240:0.2cm)..(320:1cm); 
   \node at (180:-0.4cm){\huge{$\Pi_{\ell^\prime}$}};    
   \node at (100:-0.6cm){\huge{$\Pi_{\ell^\prime+1}$}};     
        }
\end{scope}
\end{tikzpicture}
\]
\end{enumerate}
\caption{The Ptolemy conditions}
\label{04}
\end{figure}
\begin{example}
Let $n=3,p=3$. Suppose $\X$ is a
  Ptolemy diagram with $X=(2,5,1)\in\X$.  
\begin{itemize}
  \item [(1)] If the indecomposable object $Y=(4,6,3)$ belongs to $\X$, then $X$ crosses $Y$ and thus $(2,4,1)$ belongs to $\X$. 
  \item [(2)] If the indecomposable object $Z=(1,3,2)$ belongs to $\X$, then $X$ crosses $Z$ and thus $(3,5,1)$ belongs to $\X$. 
\end{itemize}
\end{example}
\subsection{Classification of torsion pairs in $\C_{n,p}$}
In this subsection, we give a complete classification of torsion pairs in $\C_{n,p}$ via Ptolemy diagrams in $\Pi^p$.
The main result of this paper is the following classification of torsion pairs in $\C_{n,p}$ via Ptolemy diagrams in $\Pi^p$.
\begin{theorem}\label{mainresult}
Let $\X$ be a subcategory of $\C_{n,p}$, and $\XX$ be the corresponding set of diagonals in $\Pi^p$. Then the following statements are equivalent.
\begin{itemize}
  \item [(1)] $\X$ is closed under extensions.
  \item [(2)] $(\X,\X^\perp)$ is a torsion pair.
  \item [(3)] $\X={^\perp(\X^\perp)}$.
  \item [(4)] $\XX$ is a Ptolemy diagram.
\end{itemize}
\end{theorem}
In order to prove the main theorem, we need some preparation.

Let $\mathfrak{U}$ be a set of diagonals in $\Pi^p$. We define $\mathfrak{U}^\bot$ to be the set consisting of two kinds of diagonals:
\begin{itemize}
  \item [(1)] $u\in\Pi^p$ such that $u$ does not cross any diagonals in $\mathfrak{U}$, or
  \item [(2)] $u=(i^\prime,j^\prime,\ell^\prime)\in\Pi^p$ such that any diagonal $v=(i,j,\ell)\in\mathfrak{U}$ crossing $u$ satisfies one of the followings:
     \begin{itemize}
  \item [(i)] $\ell=\ell^\prime$ and $1\leq i<i^\prime<j<j^\prime\leq N$.
  \item [(ii)] $\ell=\ell^\prime-1$ and $1\leq i^\prime<i<j^\prime<j\leq N$.
  \end{itemize} 
\end{itemize}
Note that when $p=1$, condition (ii) vanishes. Similarly, we define $^\bot\mathfrak{U}$ to be the set consisting of two kinds of diagonals:
\begin{itemize}
  \item [(1)] $u\in\Pi^p$ such that $u$ does not cross any diagonals in $\mathfrak{U}$, or
  \item [(2)] $u=(i^\prime,j^\prime,\ell^\prime)\in\Pi^p$ such that any diagonal $v=(i,j,\ell)\in\mathfrak{U}$ crossing $u$ satisfies one of the followings:
     \begin{itemize}
  \item [(i)] $\ell=\ell^\prime$ and $1\leq i^\prime<i<j^\prime<j\leq N$.
  \item [(ii)] $\ell=\ell^\prime+1$ and $1\leq i<i^\prime<j<j^\prime\leq N$.
  \end{itemize} 
\end{itemize}
\begin{lemma}\label{lemma1}
Let $\mathfrak{U}$ be a set of diagonals in $\Pi^p$. Then $\mathfrak{U}^\bot$ and $^\bot\mathfrak{U}$ are Ptolemy diagrams.
\end{lemma}
\begin{proof}
We need to show $\mathfrak{U}^\bot$ is a Ptolemy diagram, $^\bot\mathfrak{U}$ is a Ptolemy diagram can be proved similarly. 
\begin{itemize}
  \item [(1)] First, we need to show $\mathfrak{U}^\bot$ satisfies condition Pt1. Suppose $X=(i,j,\ell)$ and $Y=(i^\prime,j^\prime,\ell)$ are crossing diagonals in $\mathfrak{U}^\bot$ with $i<i^\prime<j<j^\prime$, and $(i^\prime,j,\ell)$ is a diagonal, we claim $(i^\prime,j,\ell)\in\mathfrak{U}^\bot$. Suppose, on the contrary, $(i^\prime,j,\ell)\notin\mathfrak{U}^\bot$, then there exists a diagonal $w\in\mathfrak{U}$ such that $w=(s,t,k)$ crosses  $(i^\prime,j,\ell)$, We have the following two cases:
      \begin{itemize}
        \item [(i)] $k=\ell$ and $i^\prime<s<j<t$. Then $w=(s,t,k)$ crosses $X=(i,j,\ell)$ with $i<s<j<t$. This is a contradiction with the fact that $w\in\mathfrak{U}$ and $X\in\mathfrak{U}^\bot$.  
        \item [(ii)] $k=\ell+1$ and $s<i^\prime<t<j$. Then $w=(s,t,k)$ crosses $Y=(i^\prime,j^\prime,\ell)$ with $s<i^\prime<t<j^\prime$. This is a contradiction with the fact that $w\in\mathfrak{U}$ and $Y\in\mathfrak{U}^\bot$. 
      \end{itemize}
      Thus we proved the claim. If $(i,j^\prime,\ell)$ is a diagonal, then $(i,j^\prime,\ell)\in\mathfrak{U}^\bot$ can be proved similarly. Dually, suppose $X=(i,j,\ell)$ and $Y=(i^\prime,j^\prime,\ell)$ are crossing diagonals in $\mathfrak{U}^\bot$ with $i^\prime<i<j^\prime<j$, similarly we can prove $(i,j^\prime,\ell)$ and $(i^\prime,j,\ell)$ are in $\mathfrak{U}^\bot$ if they are diagonals.
  \item [(2)] We need to show $\mathfrak{U}^\bot$ satisfies condition Pt2. Suppose $X=(i,j,\ell)$ and $Y=(i^\prime,j^\prime,\ell+1)$ are crossing diagonals in $\mathfrak{U}^\bot$ with $i^\prime<i<j^\prime<j$, and $(i^\prime,i,\ell+1)$ is a diagonal, we claim $(i^\prime,i,\ell+1)\in\mathfrak{U}^\bot$. Suppose, on the contrary, $(i^\prime,i,\ell+1)\notin\mathfrak{U}^\bot$, then there exists a diagonal $w\in\mathfrak{U}$ such that $w=(s,t,k)$ crossing  $(i^\prime,i,\ell+1)$, We have the following two cases:
      \begin{itemize}
        \item [(i)] $k=\ell+1$ and $i^\prime<s<i<t$. If $t<j$, then $w=(s,t,\ell+1)$ crosses $X=(i,j,\ell)$ with $s<i<t<j$, this is a contradiction with the fact that $w\in\mathfrak{U}$ and $X\in\mathfrak{U}^\bot$. If $t\geq j$, then $w=(s,t,\ell+1)$ crosses $Y=(i^\prime,j^\prime,\ell+1)$ with $i^\prime<s<j^\prime<t$, this is a contradiction with the fact that $w\in\mathfrak{U}$ and $Y\in\mathfrak{U}^\bot$.
        \item [(ii)] $k=\ell+2$ and $s<i^\prime<t<i$. Then $w=(s,t,k)$ crosses $Y=(i^\prime,j^\prime,\ell+1)$ with $s<i^\prime<t<j^\prime$. This is a contradiction with the fact that $w\in\mathfrak{U}$ and $Y\in\mathfrak{U}^\bot$. 
      \end{itemize}
      Thus we proved the claim. If $(j^\prime,j,\ell)$ is a diagonal, then $(j^\prime,j,\ell)\in\mathfrak{U}^\bot$ can be proved similarly. 
  \item [(3)] The case $\mathfrak{U}^\bot$ satisfies condition Pt3 can be proved similarly as (2).
\end{itemize}
\end{proof}

\begin{lemma}\label{lemma2}
Suppose $\mathfrak{U}$ is a set of diagonals in $\Pi^p$ and $(i,j,\ell)$ is a diagonal in $^\bot(\mathfrak{U}^\bot)$ with $1\leq i<j\leq N$. Then the following statements are hold: 
\begin{itemize}
  \item [(1)] If $i\neq 1$, then $(i-1,i+1,\ell)\notin\mathfrak{U}^\bot$.
  \item [(2)] If $j\neq N$, then $(j-1,j+1,\ell-1)\notin\mathfrak{U}^\bot$.
\end{itemize}
\end{lemma}
\begin{proof}
For $i\neq 1$, if $(i-1,i+1,\ell)\in\mathfrak{U}^\bot$, this is a contradiction with the fact that $(i,j,\ell)\in{^\bot(\mathfrak{U}^\bot)}$ and $i-1<i<i+1<j$. Similarly, for $j\neq N$, if $(j-1,j+1,\ell-1)\in\mathfrak{U}^\bot$, this is a contradiction with the fact that $(i,j,\ell)\in{^\bot(\mathfrak{U}^\bot)}$ and $i<j-1<j<j+1$.
\end{proof}
For any vertex $a$ in $\Pi^p$, set
$$(a,-,\ell)=\{(a,b,\ell)|(a,b,\ell)\text{\;is\;a\;diagonal\;in\;}\Pi^p\text{\;with\;}a<b\}.$$
Specifically, $(a,-,\ell)$ is the set of diagonals in $\Pi^p$ with $a$ as the node such that for any $(a,b,\ell)\in(a,-,\ell)$, we have $a<b$. Dually, we set
$$(-,a,\ell)=\{(c,a,\ell)|(c,a,\ell)\text{\;is\;a\;diagonal\;in\;}\Pi^p\text{\;with\;}c<a\}.$$

Note that since $(a,-,\ell)$ is a finite set, there exists a maximal integer $b$ and a minimal integer $c$ such that $(a,b,\ell)\in(a,-,\ell)$, $(a,c,\ell)\in(a,-,\ell)$. A similar result holds for the set $(-,a,\ell)$.
\begin{lemma}\label{lemma3}
Suppose $\mathfrak{U}$ is a Ptolemy diagram and $a$ is a vertex in $\Pi^p$ satisfying $(a,-,\ell)\cap\mathfrak{U}\neq\emptyset$ and $(-,a,\ell+1)\cap\mathfrak{U}\neq\emptyset$. Suppose $b$ is maximal such that $(a,b,\ell)\in\mathfrak{U}$, and $c$ is minimal such that $(c,a,\ell+1)\in\mathfrak{U}$. Then $(c,b,\ell)$ is a diagonal in $\mathfrak{U}^\bot$.
\end{lemma}
\begin{proof}
Suppose on the contrary $(c,b,\ell)\notin\mathfrak{U}^\bot$. Then then there exists a diagonal $w\in\mathfrak{U}$ such that $w=(s,t,k)$ crosses  $(c,b,\ell)$, We have the following two cases:
      \begin{itemize}
        \item [(1)] $k=\ell$ and $c<s<b<t$. We have the following three cases:
        \begin{itemize}
          \item [(i)] If $s<a$, then $w=(s,t,\ell)$ crosses $(c,a,\ell+1)$ with $c<s<a<t$. Since $\mathfrak{U}$ is a Ptolemy diagram, we have $(a,t,\ell)\in\mathfrak{U}$ by Pt2. This is a contradiction with the maximality of $b$.
          \item [(ii)] If $s=a$, then $w=(a,t,\ell)\in\mathfrak{U}$. This is a contradiction with the maximality of $b$.
          \item [(iii)] If $s>a$, then $w=(s,t,\ell)$ crosses $(a,b,\ell)$ with $a<s<b<t$. Since $\mathfrak{U}$ is a Ptolemy diagram, we have $(a,t,\ell)\in\mathfrak{U}$ by Pt1. This is a contradiction with the maximality of $b$.
        \end{itemize}
        \item [(2)] $k=\ell+1$ and $s<c<t<b$. We have the following three cases:
        \begin{itemize}
          \item [(i)] If $t<a$, then $w=(s,t,\ell+1)$ crosses $(c,a,\ell+1)$ with $s<c<t<a$. Since $\mathfrak{U}$ is a Ptolemy diagram, we have $(s,a,\ell+1)\in\mathfrak{U}$ by Pt1. This is a contradiction with the minimality of $c$.
          \item [(ii)] If $t=a$, then $w=(s,a,\ell+1)\in\mathfrak{U}$. This is a contradiction with the minimality of $c$.
          \item [(iii)] If $t>a$, then $w=(s,t,\ell+1)$ crosses $(a,b,\ell)$ with $s<a<t<b$. Since $\mathfrak{U}$ is a Ptolemy diagram, we have $(s,a,\ell+1)\in\mathfrak{U}$ by Pt3. This is a contradiction with the minimality of $c$.
        \end{itemize}
      \end{itemize}
Thus we proved $(c,b,\ell)\in\mathfrak{U}^\bot$.
\end{proof}
\begin{corollary}\label{coro1}
Suppose $\mathfrak{U}$ is a Ptolemy diagram and $a(a\neq1)$ is a vertex in $\Pi^p$ satisfying $(a,-,\ell)\cap\mathfrak{U}\neq\emptyset$ and $(-,a,\ell+1)\cap\mathfrak{U}=\emptyset$. Suppose $b$ is maximal such that $(a,b,\ell)\in\mathfrak{U}$. Then $(a-1,b,\ell)$ is a diagonal in $\mathfrak{U}^\bot$.
\end{corollary}
\begin{proof}
Since $(-,a,\ell+1)\cap\mathfrak{U}=\emptyset$, the edge $(a-1,a,\ell)$ can be seen as a degenerated diagonal of $\mathfrak{U}$. Then the result is a direct consequence of Lemma \ref{lemma3}.
\end{proof}
\begin{corollary}\label{coro2}
Suppose $\mathfrak{U}$ is a Ptolemy diagram and $a(a\neq N)$ is a vertex in $\Pi^p$ satisfying $(-,a,\ell+1)\cap\mathfrak{U}\neq\emptyset$ and $(a,-,\ell)\cap\mathfrak{U}=\emptyset$. Suppose $c$ is minimal such that $(c,a,\ell+1)\in\mathfrak{U}$. Then $(c,a+1,\ell)$ is a diagonal in $\mathfrak{U}^\bot$.
\end{corollary}
\begin{proposition}\label{prop1}
$\mathfrak{U}$ is a Ptolemy diagram if and only if $\mathfrak{U}={^\bot(\mathfrak{U}^\bot)}$.
\end{proposition}
\begin{proof}
Suppose $\mathfrak{U}={^\bot(\mathfrak{U}^\bot)}$. Then $\mathfrak{U}$ is a Ptolemy diagram by Lemma \ref{lemma1}.

Now suppose $\mathfrak{U}$ is a Ptolemy diagram. The inclusion $\mathfrak{U}\subseteq {^\bot(\mathfrak{U}^\bot)}$ is clear. (Note that this inclusion does not require $\mathfrak{U}$ to be a Ptolemy diagram.) Suppose $u=(i,j,\ell)$ is a diagonal in $^\bot(\mathfrak{U}^\bot)$, we need to show $u\in\mathfrak{U}$.

By assumption, we have $u=(i,j,\ell)\in{^\bot(\mathfrak{U}^\bot)}$. Suppose $i\neq1$, we have $(i-1,i+1,\ell)\notin\mathfrak{U}^\bot$ by Lemma \ref{lemma2}. So there exists a diagonal $w=(s,t,k)\in\mathfrak{U}$ such that $w$ crosses $(i-1,i+1,\ell)$. We have the following two cases:
\begin{itemize}
  \item [(1)] $k=\ell$ and $i-1<s<i+1<t$. This implies $w=(i,t,\ell)$. So $(i,-,\ell)\cap\mathfrak{U}\neq\emptyset$.
  \item [(2)] $k=\ell+1$ and $s<i-1<t<i+1$. This implies $w=(s,i,\ell+1)$. So $(-,i,\ell+1)\cap\mathfrak{U}\neq\emptyset$.
\end{itemize}

\smallskip

We claim: $(i,-,\ell)\cap\mathfrak{U}\neq\emptyset$. 

Suppose on the contrary $(i,-,\ell)\cap\mathfrak{U}=\emptyset$. Then we have $(-,i,\ell+1)\cap\mathfrak{U}\neq\emptyset$. Suppose $c$ is minimal such that $(c,i,\ell+1)\in\mathfrak{U}$, then $(c,i+1,\ell)\in\mathfrak{U}^\bot$ by Corollary \ref{coro2}. This is a contradiction with the fact that $u=(i,j,\ell)\in{^\bot(\mathfrak{U}^\bot)}$ and $c<i<i+1<j$. This completes the proof of our claim.

Now we have $(i,-,\ell)\cap\mathfrak{U}\neq\emptyset$. We claim: there are diagonals $(i,b,\ell)$ with $b\geq j$ and diagonals $(i,c,\ell)$ with $c\leq j$ in $\mathfrak{U}$. 

Suppose $b$ is maximal such that $(i,b,\ell)\in\mathfrak{U}$ and $b<j$. Suppose $d$ is minimal such that $(d,i,\ell+1)\in\mathfrak{U}$ (If $(-,i,\ell+1)\cap\mathfrak{U}=\emptyset$, then $d$ is taken to be $i-1$). Then $(d,b,\ell)\in\mathfrak{U}^\bot$ by Lemma \ref{lemma3}. This is a contradiction with the fact that $(d,b,\ell)$ crosses $u=(i,j,\ell)\in{^\bot(\mathfrak{U}^\bot)}$ and $d<i<b<j$. So there exist diagonals $(i,b,\ell)$ with $b\geq j$ in $\mathfrak{U}$.

Suppose $c$ is minimal such that $(i,c,\ell)\in\mathfrak{U}$. If $c>j$, then $(i+1,c,\ell-1)$ is a diagonal crossing $(i,j,\ell)$ with $i<i+1<j<c$. So we have $(i+1,c,\ell-1)\not\in\mathfrak{U}^\bot$. Then there exists a diagonal $w=(s,t,k)\in\mathfrak{U}$ crossing $(i+1,c,\ell-1)$. We have the following two cases:
\begin{itemize}
  \item [(i)] $k=\ell-1$ and $i+1<s<c<t$. So $w=(s,t,\ell-1)$ crosses $(i,c,\ell)$ and $i<s<c<t$. Since $\mathfrak{U}$ is a Ptolemy diagram, we have $(i,s,\ell)$ by Pt2. This is a contradiction with the minimality of $c$.
  \item [(ii)] $k=\ell$ and $s<i+1<t<c$. So $w=(s,t,\ell)$ crosses $(i,c,\ell)$ and $s<i<t<c$ (obviously $s\neq i$). Since $\mathfrak{U}$ is a Ptolemy diagram, we have $(i,t,\ell)$ by Pt1. This is a contradiction with the minimality of $c$.
\end{itemize}
So there exist diagonals $(i,c,\ell)$ with $c\leq j$ in $\mathfrak{U}$. This completes the proof of our claim.

\smallskip
%\vspace{\smallskipamount}

Suppose $c$ is maximal such that $(i,c,\ell)\in\mathfrak{U}$ with $c\leq j$, and suppose $b$ is minimal such that $(i,b,\ell)\in\mathfrak{U}$ with $b\geq j$. If $c=j$ or $b=j$, then $(i,j,\ell)\in\mathfrak{U}$, which completes our proof.

Now suppose $c<j$ and $b>j$, we show there is a contradiction.

By assumption, $(c,b,\ell-1)$ is a diagonal crossing $(i,j,\ell)$ with $i<c<j<b$. Then $(c,b,\ell-1)\not\in\mathfrak{U}^\bot$. So there exists $w=(s,t,k)\in\mathfrak{U}$ such that $w$ crosses $(c,b,\ell-1)$. We have the following two cases:
\begin{itemize}
  \item [(i)] $k=\ell-1$ and $c<s<b<t$. So $w=(s,t,\ell-1)$ crosses $(i,b,\ell)$ and $i<s<b<t$. Since $\mathfrak{U}$ is a Ptolemy diagram, we have $(i,s,\ell)$ by Pt2. This is a contradiction with the minimality of $b$ or the maximality of $c$.
  \item [(ii)] $k=\ell$ and $s<c<t<b$. If $s>i$ (obviously $s\neq i$), then $w=(s,t,\ell)$ crosses $(i,c,\ell)$ with $i<s<c<t$. Since $\mathfrak{U}$ is a Ptolemy diagram, we have $(i,t,\ell)$ by Pt1. This is a contradiction with the maximality of $c$. If $s<i$, then $w=(s,t,\ell)$ crosses $(i,b,\ell)$ with $s<i<t<b$. Since $\mathfrak{U}$ is a Ptolemy diagram, we have $(i,t,\ell)$ by Pt1. This is a contradiction with the minimality of $b$ or the maximality of $c$. 
\end{itemize}
If $i=1$, then $j\neq N$ (since $(i,j,\ell)$ is a diagonal) and $(j-1,j+1,\ell-1)\notin\mathfrak{U}^\bot$ by Lemma \ref{lemma2}. We can prove similarly that $(i,j,\ell)\in\mathfrak{U}$ as the case $(i-1,i+1,\ell)\notin\mathfrak{U}^\bot$.
\end{proof}

We are now ready to prove the main result Theorem \ref{mainresult}.

$\bold{Proof\; of\; Theorem\; \ref{mainresult}}$:
The equivalence of (1)-(3) is obvious by Lemma \ref{lemma0}. By Lemma \ref{lem2}, the subcategory  $\X^\perp[-1]$ (resp. {$^\perp\X[1]$}) corresponds to $\XX^\perp$ (resp. $^\perp\XX$). The equivalence of (3) and (4) is a direct consequence of Proposition \ref{prop1}.

\section{Applications}

\subsection{Classification of rigid objects and cluster tilting objects in $\C_{n,p}$ }
Firstly, we recall the definition of rigid subcategories and cluster tilting subcategories in $\C_{n,p}$.

\begin{definition}[\cite{KR,T,Z,IY,J}]
Let $\X$ be a subcategory of $\C_{n,p}$.
\begin{itemize}
  \item [(1)] The subcategory $\X$ is called rigid if $\Ext_{\C_{n,p}}^{1}(\X,\X)=0$. The subcategory $\X$ is called maximal rigid if $\X$ is maximal with respect to this property, i.e. if $\Ext_{\C_{n,p}}^{1}(\X\oplus\add M,\X\oplus\add M)=0$, then $M\in\X$. An object $X$ is called a rigid object if $\add X$ is rigid, $X$ is called a maximal rigid object if $\add X$ is maximal rigid.
  \item [(2)] $\X$ is called cluster tilting if $X\in\X$ if and only if
      $\Ext_{\C_{n,p}}^{1}(X,\X)=0$ if and only if $\Ext_{\C_{n,p}}^{1}(\X,X)=0$.
      An object $X$ is called a cluster tilting object if $\add\;X$ is cluster tilting.
\end{itemize}
We point out the fact that for an object $X$ of $\C_{n,p}$, $X$ is cluster tilting if and only if $X$ is  maximal rigid \cite[Theorem 2.6]{ZZ}.
\end{definition}
\begin{definition}
Let $\mathfrak{U}$ be a set of diagonals in $\Pi^p$, and $\mathfrak{U}=\mathfrak{U}\cap\Pi^p=\mathfrak{U}_1\cup \mathfrak{U}_2\cup\ldots\cup \mathfrak{U}_p$ where $\mathfrak{U}_k\subseteq\Pi_k$, $1\leq k\leq p$.
\begin{itemize}
  \item [(1)] $\mathfrak{U}_k$ is called a partial triangulation of $\Pi_k$ if $\mathfrak{U}_k$ is a set of non-crossing diagonals in $\Pi_k$.
  \item [(2)] $\mathfrak{U}_k$ is called a triangulation of $\Pi_k$ if $\mathfrak{U}_k$ is a maximal set of non-crossing diagonals in $\Pi_k$.
\end{itemize}
\end{definition}
\begin{proposition}\label{prop2}
Let $\X$ be a subcategory of $\C_{n,p}$, and $\XX$ be the corresponding set of diagonals in $\Pi^p$ such that $\XX=\XX_1\cup \XX_2\cup\ldots\cup \XX_p$. Then the following statements are equivalent.
\begin{itemize}
  \item [(1)] $\X$ is rigid.
  \item [(2)] For all $1\leq k\leq p$, $\XX_k$ is a partial triangulation such that the diagonals in $\XX_k$ and $\XX_{k+1}$ do not cross, and the diagonals in $\XX_k$ and $\XX_{k-1}$ do not cross.
  \item [(3)] $(\X,\X^\perp)$ is a torsion pair in $\C_{n,p}$ such that $\X\subseteq\X^\perp[-1]\cap{^\perp\X[1]}$.
\end{itemize}
\end{proposition}
\begin{proof}
The equivalence of $(1)$ and $(2)$ is clear by Lemma \ref{lem2}, and the equivalence of $(1)$ and $(3)$ is a direct consequence of Theorem \ref{mainresult}.
\end{proof}
\begin{corollary}\cite[Proposition 5.4]{L1}\label{coro3}
Let $\X$ be a subcategory of $\C_{n,p}$, and $\XX$ be the corresponding set of diagonals in $\Pi^p$ such that $\XX=\XX_1\cup \XX_2\cup\ldots\cup \XX_p$. Then the following statements are equivalent.
\begin{itemize}
  \item [(1)] $\X$ is cluster tilting.
  \item [(2)] $\XX_1$ is a triangulation of $\Pi$, and for all $2\leq k\leq p$, $\XX_{k}=\rho(\XX_{k-1})$.
  \item [(3)] $(\X,\X^\perp)$ is a torsion pair in $\C_{n,p}$ such that $\X=\X^\perp[-1]={^\perp\X[1]}$.
\end{itemize}
\end{corollary}
\begin{corollary}\cite[Corollary 5.5]{L1}\label{coro4}
Any cluster tilting object in $\C_{n,p}$ contains $pn$ pairwise non-isomorphic summands.
\end{corollary}
\begin{example}
Let $n=3,p=4$, $X=\add((1,3,1)\oplus(2,4,3))$ and $Y=\add((1,3,1)\oplus(1,4,1)\oplus(1,5,1)\oplus(1,3,2)\oplus(1,4,2)\oplus(1,5,2)\oplus(1,3,3)\oplus(1,4,3)
\oplus(1,5,3)\oplus(1,3,4)\oplus(1,4,4)\oplus(1,5,4))$. Then $X$ is a rigid subcategory of $\C_{3,4}$ by Proposition \ref{prop2}, and $Y$ is a cluster tilting subcategory of $\C_{n,p}$ by Corollary \ref{coro3}.

Note that the rigid subcategory $X$ can not become a cluster tilting subcategory by adding some rigid objects, i.e.  $X$ is not a subcategory of any cluster tilting subcategory.
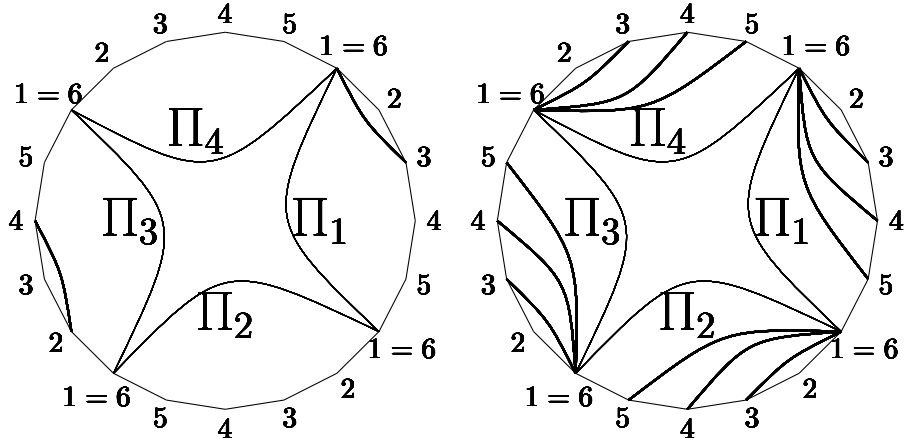
\begin{figure}
\begin{center}
\begin{tikzpicture}[scale=2.5]
\begin{scope}
    \foreach \x in {-18,0,18,36,54,72,90,108,126,144,162,180,198,216,234,252,270,288,306,324}{
        \draw (\x:1 cm) -- (\x + 18: 1cm) -- cycle;
        \draw [thick](54:1cm)..controls(36:0.9cm)..(18:1cm);
        \draw [thick](216:1cm)..controls(198:0.9cm)..(180:1cm);
        %\draw [thick][dashed](20:1cm)--(60:1cm);
        %\draw [thick][dashed](60:1cm)--(-40:1cm);
        %\draw [thick][dashed](0:1cm)--(-40:1cm);
        %\draw [thick][dashed](20:1cm)..controls(10:0.97cm)..(0:1cm);
        \node at (36:1.1cm) {$2$};
        \node at (18:1.1cm) {$3$};
        \node at (0:1.1cm) {$4$};
        \node at (-18:1.1cm) {$5$};
        \node at (324:1.15cm) {$1=6$};
        \node at (306:1.1cm) {$2$};
        \node at (288:1.1cm) {$3$};
        \node at (270:1.1cm) {$4$};
        \node at (252:1.1cm) {$5$};
        \node at (234:1.15cm) {$1=6$};
        \node at (216:1.1cm) {$2$};
        \node at (198:1.1cm) {$3$};
        \node at (180:1.1cm) {$4$};
        \node at (162:1.1cm) {$5$};
        \node at (144:1.15cm) {$1=6$};
        \node at (126:1.1cm) {$2$};
        \node at (108:1.1cm) {$3$};
        \node at (90:1.1cm) {$4$};
        \node at (72:1.1cm) {$5$};
        \node at (54:1.15cm) {$1=6$};
 \draw (54:1cm)..controls(0:0.2cm)..(324:1cm); 
  \draw (324:1cm)..controls(270:0.2cm)..(234:1cm);  
   \draw (234:1cm)..controls(180:0.2cm)..(144:1cm); 
   \draw (144:1cm)..controls(108:0.2cm)..(54:1cm);
   \node at (0:0.5cm){\huge{$\Pi_1$}};  
   \node at (270:0.5cm){\huge{$\Pi_2$}}; 
     \node at (180:0.5cm){\huge{$\Pi_3$}};  
   \node at (108:0.5cm){\huge{$\Pi_4$}};    
        }
\end{scope}
\end{tikzpicture}
\begin{tikzpicture}[scale=2.5][auto]
\begin{scope}
    \foreach \x in {-18,0,18,36,54,72,90,108,126,144,162,180,198,216,234,252,270,288,306,324}{
        \draw (\x:1 cm) -- (\x + 18: 1cm) -- cycle;
         \draw [thick](54:1cm)..controls(36:0.9cm)..(18:1cm);
         \draw [thick](54:1cm)..controls(27:0.7cm)..(0:1cm);
         \draw [thick](54:1cm)..controls(18:0.6cm)..(-18:1cm);
          \draw [thick](234:1cm)..controls(216:0.9cm)..(198:1cm);
         \draw [thick](234:1cm)..controls(207:0.7cm)..(180:1cm);
         \draw [thick](234:1cm)..controls(198:0.6cm)..(162:1cm);
         \draw [thick](144:1cm)..controls(126:0.9cm)..(108:1cm);
         \draw [thick](144:1cm)..controls(117:0.7cm)..(90:1cm);
         \draw [thick](144:1cm)..controls(108:0.6cm)..(72:1cm);
          \draw [thick](324:1cm)..controls(306:0.9cm)..(288:1cm);
         \draw [thick](324:1cm)..controls(297:0.7cm)..(270:1cm);
         \draw [thick](324:1cm)..controls(288:0.6cm)..(252:1cm);
        %\draw [thick][dashed](20:1cm)--(60:1cm);
        %\draw [thick][dashed](60:1cm)--(-40:1cm);
        %\draw [thick][dashed](0:1cm)--(-40:1cm);
        %\draw [thick][dashed](20:1cm)..controls(10:0.97cm)..(0:1cm);
        \node at (36:1.1cm) {$2$};
        \node at (18:1.1cm) {$3$};
        \node at (0:1.1cm) {$4$};
        \node at (-18:1.1cm) {$5$};
        \node at (324:1.15cm) {$1=6$};
        \node at (306:1.1cm) {$2$};
        \node at (288:1.1cm) {$3$};
        \node at (270:1.1cm) {$4$};
        \node at (252:1.1cm) {$5$};
        \node at (234:1.15cm) {$1=6$};
        \node at (216:1.1cm) {$2$};
        \node at (198:1.1cm) {$3$};
        \node at (180:1.1cm) {$4$};
        \node at (162:1.1cm) {$5$};
        \node at (144:1.15cm) {$1=6$};
        \node at (126:1.1cm) {$2$};
        \node at (108:1.1cm) {$3$};
        \node at (90:1.1cm) {$4$};
        \node at (72:1.1cm) {$5$};
        \node at (54:1.15cm) {$1=6$};
 \draw (54:1cm)..controls(0:0.2cm)..(324:1cm); 
  \draw (324:1cm)..controls(270:0.2cm)..(234:1cm);  
   \draw (234:1cm)..controls(180:0.2cm)..(144:1cm); 
   \draw (144:1cm)..controls(108:0.2cm)..(54:1cm);
   \node at (0:0.5cm){\huge{$\Pi_1$}};  
   \node at (270:0.5cm){\huge{$\Pi_2$}}; 
     \node at (180:0.5cm){\huge{$\Pi_3$}};  
   \node at (108:0.5cm){\huge{$\Pi_4$}};    
        }
\end{scope}
\end{tikzpicture}
\caption{The subcategories $X$ (on the left) and $Y$ (on the right) of $\C_{n,p}$}
\label{001}
\end{center}
\end{figure}
\end{example}

\subsection{Relationship with the cluster category of type $A_{n}$}
When $p=1$, $\C_{n,1}$ (denote by $\C_{n}$) is the classical cluster category of type $A_{n}$. In this subsection, we use our classification of torsion pairs in
$\C_{n,p}$ to recover one main result in \cite{HJR1}, which gives a correspondence between torsion pairs and Ptolemy diagrams of the $(n+3)$-gon.

Note that when $p=1$, the definition of Ptolemy diagram in Definition \ref{c2} reduces to the following case.
\begin{definition}
Suppose $p=1$. Let $\mathfrak{U}$  be a set of diagonals in $\Pi$. Then $\mathfrak{U}$ is called a
Ptolemy diagram if it has the following property: when $X=(i,j,1)$ and $Y=(i^\prime,j^\prime,1)$
are crossing diagonals in $\Pi$ (so $i<i^\prime<j<j^\prime$ or $i^\prime<i<j^\prime<j$ by Definition \ref{cross}), then 
\begin{itemize}
\item [(i)] those of $(i,i^\prime,1)$, $(i,j^\prime,1)$, $(i^\prime,j,1)$, $(j,j^\prime,1)$ which are diagonals of $\Pi$ are also in $\mathfrak{U}$ if $ i<i^\prime<j<j^\prime$.
\item [(ii)] those of $(i^\prime,i,1)$, $(i^\prime,j,1)$, $(i,j^\prime,1)$, $(j^\prime,j,1)$ which are diagonals of $\Pi$ are also in $\mathfrak{U}$ if $i^\prime<i<j^\prime<j$.
\end{itemize}
\end{definition}
In this case, the definition of Ptolemy diagram is compatible with Definition 2.1 in \cite{HJR1}. So we have the following result by Theorem \ref{mainresult}.
\begin{corollary}\cite[Theorem A]{HJR1}
Let $\X$ be a subcategory of $\C_{n}$, and $\XX$ be the corresponding set of diagonals in the $(n+3)$-gon. Then the following statements are equivalent:
\begin{itemize}
  \item [(1)] $(\X,\X^\perp)$ is a torsion pair of $\C_{n}$.
  \item [(2)] $\XX$ is a Ptolemy diagram.
\end{itemize}
\end{corollary}

\bigskip

%$\bold {Acknowlegement}$: We would like to thank the referees for their helpful comments and suggestions.

%\end{example}

\end{document}